\documentclass[12pt]{extarticle}
\usepackage{amsmath, amsthm, amssymb, color}
\usepackage[colorlinks=true,linkcolor=blue,urlcolor=blue]{hyperref}
\usepackage{graphicx}
\usepackage{caption}
\usepackage{mathtools}
\usepackage{enumerate}
\usepackage{verbatim}
\usepackage{tikz,tikz-cd,tikz-3dplot}
\usepackage{amssymb}
\usetikzlibrary{matrix}
\usetikzlibrary{arrows}
\usepackage{algorithm}
\usepackage[noend]{algpseudocode}
\usepackage{caption}
\usepackage[normalem]{ulem}
\usepackage{subcaption}
\oddsidemargin \evensidemargin
\usepackage{makecell}
\usepackage{array}
\usepackage{enumitem}
\newtheorem{theorem}{Theorem}[section]
\newtheorem{proposition}[theorem]{Proposition}

\theoremstyle{definition}
\newtheorem{definition}[theorem]{Definition}

\newtheorem{examplex}[theorem]{Example}
\newenvironment{example}
{\pushQED{\qed}\examplex}
{\popQED\endexamplex}

\makeatletter
\def\Ddots{\mathinner{\mkern1mu\raise\p@
\vbox{\kern7\p@\hbox{.}}\mkern2mu
\raise4\p@\hbox{.}\mkern2mu\raise7\p@\hbox{.}\mkern1mu}}
\makeatother

\newcommand{\RR}{\mathbb{R}}
\newcommand{\QQ}{\mathbb{Q}}

\title{\bf Logarithmically Sparse Symmetric Matrices}

\author{Dmitrii Pavlov}
\date{}
\begin{document}
\maketitle

\begin{abstract}
\noindent
A positive definite matrix is called logarithmically sparse if its matrix logarithm has many zero entries. Such matrices play a significant role in high-dimensional statistics and semidefinite optimization. In this paper, logarithmically sparse matrices are studied from the point of view of computational algebraic geometry: we present a formula for the dimension of the Zariski closure of a set of matrices with a given logarithmic sparsity pattern, give a degree bound for this variety and develop implicitization algorithms that allow to find its defining equations. We illustrate our approach with numerous examples.
\end{abstract}

\section{Introduction}

Logarithmically sparse symmetric matrices are positive definite matrices for which the matrix logarithm is sparse. Such matrices arise in high-dimensional statistics \cite{BatteyFirst}, where structural assumptions about covariance matrices are necessary for giving consistent estimators, and sparsity assumptions are natural to make. Moreover, once the sparsity pattern is fixed, the corresponding set of logarithmically sparse matrices forms a Gibbs manifold \cite{GM}. As we recall in Section 2, this is a manifold obtained by applying the matrix exponential to a linear system of symmetric matrices (LSSM), here defined by the sparsity pattern. Gibbs manifolds play an important role in convex optimization \cite[Section 5]{GM}. 

From the point of view of practical computations, it might be challenging to tell exactly whether a given matrix satisfies a given logarithmic sparsity pattern. Checking whether a given polynomial equation holds on the matrix is often much easier. This motivates studying Zariski closures of families of logarithmically sparse matrices, i.e. common zero sets of polynomials that vanish on such families. Such Zariski closures are examples of Gibbs varieties. 

In this paper we study Gibbs varieties that arise as Zariski closures of sets of logarithmically sparse symmetric matrices. We explain how those can be encoded by graphs, give a formula for their dimension and show that in practice it can be computed using simple linear algebra. We present a numerical and a symbolic algorithm for finding their defining equations. We also investigate how graph colourings can affect the corresponding Gibbs variety. In addition, we prove some general results about Gibbs varieties. In particular, we give an upper bound for the degree of a Gibbs variety in the case when the eigenvalues of the corresponding LSSM are~$\mathbb{Q}$-linearly independent and show that Gibbs varieties of permutation invariant LSSMs inherit a certain kind of symmetry. 

This paper is organized as follows. In Section 2, we define Gibbs manifolds and Gibbs varieties, the geometric objects needed for our research, present a formula for the dimension and an upper bound for the degree of Gibbs varieties, study symmetries of their defining equations and suggest a numerical implicitization algorithm. In Section 3, we give a formal definition of logarithmic sparsity, explain how it can be encoded by graphs and discuss the special properties of Gibbs varieties defined by logarithmic sparsity. In Section 4, we study families of logarithmically sparse matrices that arise from trees. In Section 5, we study coloured logarithmic sparsity conditions. Section 6 features a symbolic implicitization algorithm for Gibbs varieties defined by logarithmic sparsity. Finally, Section 7 contains a discussion on the practical relevance of logarithmic sparsity in statistics and optimization.

\section{Gibbs Manifolds and Gibbs Varieties}

Let~$\mathbb{S}^n$ denote the space of $n\times n$ symmetric matrices. This is a real vector space of dimension~$\binom{n+1}{2}$. The cone of positive semidefinite $n\times n$ matrices will be denoted by~$\mathbb{S}^n_+$.

The matrix exponential function is defined
by the usual power series, which converges for all  real and complex $n \times n$ matrices.
It maps symmetric matrices to positive definite symmetric matrices. 
The zero matrix $0_n$ is mapped to the identity matrix ${\rm id}_n$. We write
$$ {\rm exp} \,\,:\, \mathbb{S}^n \rightarrow {\rm int} (\mathbb{S}^n_+)\,, \,\,
X \,\mapsto \, \sum_{i=0}^\infty\, \frac{1}{i !} \,X^i . $$
This map is invertible, with the inverse being the matrix logarithm function, given by the series
$$ {\rm log}\,\, :\, {\rm int} (\mathbb{S}^n_+) \rightarrow \mathbb{S}^n \,,\,\,
Y \,\mapsto \, \sum_{j=1}^\infty \frac{(-1)^{j-1}}{j} \,( \,Y - {\rm id}_n)^j.
$$

We next introduce the geometric objects that will play a crucial role in this article. We fix $d$ linearly independent matrices 
$A_1, A_2,\ldots,A_d$ in $\mathbb{S}^n$.
 We write $\mathcal{L}$ for $\mathrm{span}_{\mathbb{R}}(A_1, \ldots, A_d)\,$, a linear subspace 
 of the vector space $\mathbb{S}^n \simeq \RR^{\binom{n+1}{2}}$. Thus, $\mathcal{L}$ is a \emph{linear space~of symmetric matrices} (LSSM). 
We are interested in the image 
of $\mathcal{L}$ under the exponential map:

\begin{definition}
The \emph{Gibbs manifold} $\mathrm{GM}(\mathcal{L})$ of $\mathcal{L}$ is the $d$-dimensional manifold $\mathrm{exp}(\mathcal{L}) \subset \mathbb{S}^n_+$.
\end{definition}

This is indeed a $d$-dimensional manifold inside the convex cone $\mathbb{S}^n_+$. It is diffeomorphic to
 $\mathcal{L} \simeq \RR^d$, with the diffeomorphism given by
 the exponential map and the logarithm map.

In some special cases, the Gibbs manifold is semi-algebraic, namely it is the intersection of an algebraic variety with the PSD cone.
However, this fails in general.
It is still interesting to ask which polynomial relations hold between the entries of any matrix in $\mathrm{GM}(\mathcal{L})$. This motivates the following definition. 

\begin{definition}
The \emph{Gibbs variety} $\mathrm{GV}(\mathcal{L})$ of $\mathcal{L}$ is the Zariski closure of $\mathrm{GM}(\mathcal{L})$ in $\mathbb{C}^{\binom{n+1}{2}}$.
\end{definition}

Any LSSM can be written in the form~$\mathcal{L} = \{y_1A_1+\ldots+y_dA_d|y_i\in\mathbb{R}\}$ and therefore can be identified with a matrix with entries in~$\mathbb{R}(y_1,\ldots,y_d)$.
The eigenvalues of this matrix are elements of the algebraic closure~$\overline{\mathbb{R}(y_1,\ldots,y_d)}$ and will be referred to as the eigenvalues of the corresponding LSSM. 
It is known that~$\mathrm{GV}(\mathcal{L})$ is irreducible and unirational under the assumption that the eigenvalues of~$\mathcal{L}$ are $\mathbb{Q}$-linearly independent and~$\mathcal{L}$ is defined over~$\mathbb{Q}$ \cite[Theorem 3.6]{GM}.

In this section we extend the results of \cite{GM} that apply to any LSSM~$\mathcal{L}$. We start with studying the symmetries of the defining equations of~$\mathrm{GV}(\mathcal{L})$. 

We consider the tuple of variables~$\mathbf{x} = \{x_{ij} | 1\leqslant i \leqslant j \leqslant n\}$. An element $\sigma$ of the symmetric group~$S_n$ acts on the polynomial ring~$\RR[{\bf x}]$ by sending~$x_{ij}$ to $x_{\sigma (i) \sigma (j)}$ for $1\leqslant i \leqslant j \leqslant n$ (we identify the variables $x_{ij}$ and $x_{ji}$). We will also consider the action of $S_n$ on~$\mathbb{S}^n$ by simultaneously permuting rows and columns of a matrix.  

\begin{proposition}
    Let~$\mathcal{L}$ be an LSSM of $n \times n$ matrices that is invariant under the action of~$\sigma \in S_n$. Then the ideal~$I(\mathrm{\rm GV}(\mathcal{L}))$ of the corresponding Gibbs variety is also invariant under the action of~$\sigma$. 
\end{proposition}

\begin{proof}
    To prove the Proposition, it suffices to show that if~$B \in \mathcal{L}$ is obtained from~$A \in \mathcal{L}$ by simultaneously permuting rows and columns, then~$\exp{(B)}$ is obtained from~$\exp{(A)}$ in the same way. Since~$\exp{(B)}$ is a formal power series in~$B$, it suffices to show that $B^k$ is obtained from~$A^k$ by simultaneously permuting rows and columns for any non-negative integer~$k$. The latter fact immediately follows from the matrix multiplication formula.  
\end{proof}

\begin{example} \label{ex:sym}
Consider the LSSM
$$\mathcal{L} = \left\{\begin{pmatrix} y_1+y_2+y_3 & y_1 & y_2\\ y_1 & y_1+y_2+y_3 & y_3\\ y_2 & y_3 & y_1 + y_2 + y_3 \end{pmatrix} \bigg| y_1, y_2, y_3 \in \mathbb{R} \right\}.$$
The transposition~$\sigma = (1 2) \in S_3$ acts on~$\mathbb{S}^3$ in the following way:
$$\begin{pmatrix} x_{11} & x_{12} & x_{13}\\ x_{12} & x_{22} & x_{23}\\ x_{13} & x_{23} & x_{33} \end{pmatrix} \mapsto \begin{pmatrix} x_{22} & x_{12} & x_{23}\\ x_{12} & x_{11} & x_{13}\\ x_{23} & x_{13} & x_{33} \end{pmatrix}.$$
This action restricts to a linear automorphism of~$\mathcal{L}$. 
The Gibbs variety of~$\mathcal{L}$ is a hypersurface in~$\mathbb{C}^6$ whose prime ideal is generated by a single polynomial
\begin{multline*}
    p(x_{11},x_{12},x_{13},x_{22},x_{23},x_{33}) = (x_{11}-x_{22})(x_{11}-x_{33})(x_{22}-x_{33}) -\\- x_{33}(x_{13}^2-x_{23}^2) - x_{22}(x_{23}^2-x_{12}^2) - x_{11}(x_{12}^2 -x_{13}^2).
\end{multline*}
The action of~$\sigma$ on~$\mathbb{C}[x_{11},x_{12},x_{13},x_{22},x_{23},x_{33}]$ sends $p$ to $-p$ and therefore does not change the ideal. 
\end{example}

\begin{definition}
    Let~$A$ be an~$n\times n$ matrix, and~$\mathcal{L}$ be an LSSM of~$n \times n$ matrices. The~\emph{centralizer}~$C(A)$ of~$A$ is the set of all matrices that commute with~$A$. The~\emph{$\mathcal{L}$-centralizer}~$C_\mathcal{L}(A)$ of~$A$ is~$C(A) \cap \mathcal{L}$. 
\end{definition}

The following is an extension of \cite[Theorem 2.4]{GM}.

\begin{theorem} \label{thm:dimeq}
    Let~$\mathcal{L}$ be an LSSM of $n\times n$ matrices of dimension~$d$. Let~$k$ be the dimension of the $\mathcal{L}$-centralizer of a generic element in $\mathcal{L}$ and~$m$ the dimension of the~$\QQ$-linear space spanned by the eigenvalues of~$\mathcal{L}$. Then $\dim \mathrm{GV}(\mathcal{L}) = m + d - k$. 
\end{theorem}

\begin{proof}
    It follows from the proof of \cite[Theorem 4.6]{GM} that the dimension of a generic fiber of the map~$\phi$ that parametrizes the Gibbs variety is equal to the dimension of the centralizer of a generic element in this fiber, i.e. to~$k$. The domain of~$\phi$ is irreducible and has dimension $m+d$. Thus, by fiber dimension theorem \cite[Exercise II.3.22]{Hartshorne}, $\dim \mathrm{GV}(\mathcal{L}) = m+d-k$.
\end{proof}

Note that when~$m=k$, we have~$\dim \mathrm{GV}(\mathcal{L}) = d$ and therefore the Gibbs manifold is the positive part of the Gibbs variety, i.e.~$\mathrm{GM}(\mathcal{L}) = \mathrm{GV}(\mathcal{L}) \cap \mathbb{S}_n^+$. In particular, in this case the Gibbs manifold is a semialgebraic set. This is the case, for instance, for the LSSM of all diagonal matrices \cite[Theorem 2.7]{GM}.

We now give a degree bound for the Gibbs variety of an LSSM~$\mathcal{L}$. In what follows,~$\mathbb{V}(I)$ denotes the variety in~$\mathbb{C}^{\binom{n+1}{2}}$ defined by the ideal~$I \subseteq \mathbb{C}[\mathbf{x}]$.

\begin{proposition} \label{prop:deg}
    Let~$\mathcal{L}$ be an LSSM of $n\times n$ matrices with~$\mathbb{Q}$-linearly independent eigenvalues. Then $\deg \mathrm{GV} (\mathcal{L}) \leqslant n^{\binom{n+1}{2}+2n}$. 
\end{proposition}

\begin{proof}
    By \cite[Algorithm 1]{GM}, the prime ideal~$J$ of $\mathrm{GV}(\mathcal{L})$ is obtained by elimination from the ideal~$I$ generated by polynomials of degree at most~$n$. Therefore,~$\deg \mathrm{GV}(\mathcal{L}) = \deg \mathbb{V}(J) \leqslant \deg \mathbb{V}(I)$. The variety~$\mathbb{V}(I)$ lives in the affine space of dimension~$\binom{n+1}{2} + 2n + d$, where~$d=\dim \mathcal{L}$. Note that~$\dim \mathbb{V}(I) \geqslant \dim \mathcal{L}$ and thus $\operatorname{codim} \mathbb{V}(I) \leqslant \binom{n+1}{2}+2n$. Therefore, by B\'{e}zout's theorem, we have~$\deg \mathbb{V}(I) \leqslant n^{\binom{n+1}{2}+2n}$, which proves the Proposition.  
\end{proof}

As we will see below, the bound from Proposition \ref{prop:deg} is usually pessimistic. 

Once the degree of the Gibbs variety is known, one can use numerical techniques to find its defining equations. In general, this allows to compute ideals of Gibbs varieties that are infeasible for symbolic algorithms.

We now present Algorithm \ref{alg:num} for finding the equations of the Gibbs variety numerically. We write~$\langle P \rangle$ for the ideal generated by~$P\subseteq \mathbb{C}[\mathbf{x}]$.

\begin{algorithm}
\caption{Numerical implicitization of Gibbs varieties of~known degree} \label{alg:num}
\begin{description}[itemsep=0pt]
\item[Input ] 
    An LSSM~$\mathcal{L}$ given as an $\RR$-span of~$d$ linearly independent matrices~$A_1,\ldots,A_d$, degree~$k$ of~$\mathrm{GV}(\mathcal{L})$;
\item[Output ] A set of equations that define~$\mathrm{GV}(\mathcal{L})$ set-theoretically. 
\end{description}

\begin{enumerate}[label = \textbf{(S\arabic*)}, leftmargin=*, align=left, labelsep=2pt, itemsep=0pt]

    \item \textbf{Require} $\mathcal{L}$ has~$\mathbb{Q}$-linearly independent eigenvalues.
    \item Set~$I:=\{0\}$, $l=1$,~$N = \binom{n+1}{2}$ 
    
    \item \textbf{For} $l=1$ to~$k$ do
    \begin{enumerate}
        \item Pick~$M > \binom{N+l-1}{l}$ random samples in~$\mathcal{L}$

        \item Let~$E$ be the set of matrix exponentials of the~$M$ picked samples

        \item Construct a Vandermonde matrix~$A$ by evaluating all monomials of degree~$l$ on the elements of~$E$  

        \item Let~$I_l$ be the basis of $\ker A$
        
        \item $I:=\langle I \cup I_l \rangle$
    \end{enumerate}
    
    \item \textbf{Return} a set of generators of~$I$. 
\end{enumerate}
\end{algorithm}

Unfortunately, the degree upper bound in Proposition \ref{prop:deg} restricts the practical applicability of this algorithm to~$n\leqslant 3$. However, if the Gibbs variety is a hypersurface, then the algorithm can terminate immediately after finding a single algebraic equation. The degree of this equation is usually much lower than the degree bound in Proposition \ref{prop:deg} (for instance, the Gibbs variety in Example \ref{ex:sym} is defined by a cubic, while the bound from Proposition \ref{prop:deg} is equal to $3^{12}$) and therefore the defining equation can be found with this algorithm for larger~$n$. 

Although Algorithm \ref{alg:num} uses floating point computations, for LSSMs defined over~$\mathbb{Q}$ it can be adapted to give exact equations. This can be done using built-in commands in computer algebra systems, e.g. \texttt{rationalize} in \textsc{Julia}. Correctness of the rationalization procedure can be checked by plugging a parametrization of the Gibbs variety into the resulting equations.

\section{Logarithmic sparsity patterns}
Every set~$S \subseteq \{(i,j)|1 \leqslant i \leqslant j \leqslant n\}$ defines a sparsity pattern on symmetric matrices in the following way.

\begin{definition}
    We say that~$A = \{a_{ij}\} \in \mathbb{S}^n$ \emph{satisfies the sparsity condition given by}~$S$ if~$a_{ij}=0$ for all~$(i,j)\in S$. The set of all symmetric matrices satisfying the sparsity condition given by~$S$ forms an LSSM with a basis~$\left\{\frac{E_{ij}+E_{ji}}{2} \big| (i,j)\not\in S\right\}$, where~$E_{ij}$ is a matrix unit, i.e. a matrix with only one non-zero entry, which is equal to~$1$, at the position~$(i,j)$. We will denote this LSSM by~$\mathcal{L}_S$ and write~$\mathcal{L}_S = \left\{ \sum\limits_{(i,j)\not\in S} y_{ij} \left(\frac{E_{ij}+E_{ji}}{2}\right)\big| y_{ij}\in\mathbb{R} \right\}$. 
\end{definition}

\begin{example}
    Let~$n=4$ and $S = \{(1,2),(1,3),(2,4)\}$. The corresponding LSSM
    $$\mathcal{L}_S = \begin{pmatrix}
        y_{11} & 0 & 0 & y_{14}\\
        0 & y_{22} & y_{23} & 0\\
        0 & y_{23} & y_{33} & y_{34}\\
        y_{14} & 0 & y_{34} & y_{44}
    \end{pmatrix}$$
    is cut out by the equations~$y_{12} = y_{13} = y_{24} = 0$. 
\end{example}

\begin{definition}
   We say that~$A \in \mathrm{int}{(\mathbb{S}_+^n)}$ \emph{satisfies the logarithmic sparsity condition given by}~$S$ if~$\log{A} \in \mathcal{L}_S$. 
\end{definition}

Sparsity patterns can be encoded by graphs, which allows to study them from a combinatorial point of view. Namely, to any simple undirected graph~$G$ on~$n$ nodes we associate a set~$S_G  \subseteq \{(i,j)|1 \leqslant i \leqslant j \leqslant n\}$ as follows: $(i,j) \in S_G$ if and only if there is no edge between the nodes~$i$ and~$j$ in $G$. In this case we will also denote the corresponding LSSM by~$\mathcal{L}_G$. Note that if~$G$ has~$n$ nodes and~$e$ edges, then~$\dim \mathcal{L}_G = n+e$. 

We are interested in an algebraic description of the set of matrices that satisfy a logarithmic sparsity pattern given by~$G$. This set of matrices is precisely the Gibbs manifold of~$\mathcal{L}_G$. Since disconnected graphs correspond to LSSMs with block-diagonal structure and block-diagonal matrices are exponentiated block-wise, we will only consider the case of connected~$G$.  

LSSMs given by graphs are nice in a sense that finding the dimension of their Gibbs varieties can be reduced to a simple linear algebra procedure of computing matrix centralizers. This is justified by the following result. 

\begin{proposition} \label{prop:Galois}
    Let~$\mathcal{L}_G$ be an LSSM given by a simple connected graph~$G$ on~$n$ nodes. Then its eigenvalues are~$\QQ$-linearly independent. 
\end{proposition}

\begin{proof}
    By specializing the variables~$y_{ij}$ to zero for~$i\neq j$ and the variables~$y_{ii}$ to~$n$ $\QQ$-linearly independent algebraic numbers, we obtain a diagonal element of~$\mathcal{L}$ whose eigenvalues are linearly independent over~$\mathbb{Q}$. This immediately implies $\mathbb{Q}$-linear independence of the eigenvalues of~$\mathcal{L}$.
\end{proof}

We now address the question of computing the~$\mathcal{L}_G$-centralizer of a generic element~$A \in \mathcal{L}_G$. 
One way to do this is by straightforwardly solving the system of~$\binom{n}{2}$ equations~$XA-AX = 0$ in the variables~$x_{ij}$ over the field~$\mathbb{Q}(a_{ij})$, where~$x_{ij}$ are the entries of~$X \in \mathcal{L}_G$ and~$a_{ij}$ are the entries of~$A$.
However, there is a way to give a more explicit description of the $\mathcal{L}_G$-centralizer.

Note that by Proposition \ref{prop:Galois} the eigenvalues of~$\mathcal{L}_G$ are~$\QQ$-linearly independent.
In particular, this implies that the eigenvalues of~$A \in \mathcal{L}$ are generically distinct and that~$A$ is generically non-derogatory (\cite[Definition 1.4.4]{MA}).
Therefore, by \cite[Theorem 4.4.17, Corollary 4.4.18]{TMA}, we have $C(A) = \mathrm{span}_\RR (\mathrm{id}_n,A,\ldots,A^{n-1})$, where~$\mathrm{id}_n$ is the~$n\times n$ identity matrix.
Hence, finding $C_\mathcal{L}(A)$ reduces to intersecting $\mathrm{span}_\RR (\mathrm{id}_n,A,\ldots,A^{n-1})$ with~$\mathcal{L}$.
Such an intersection can be found by solving a system of linear equations~$p_0\mathrm{id}_n + p_1A + \ldots + p_{n-1}A^{n-1} = \sum \limits_{(i,j)\not\in S_G} c_{ij}E_{ij}$ in the variables~$p_0,\ldots,p_{n-1},c_{ij}$. 
Since~$\mathrm{id}_n$ and~$A$ are both in~$\mathcal{L}$, the intersection is at least two-dimensional and we arrive at the following proposition. 

\begin{proposition} \label{prop:upbound}
    Let~$G$ be a simple connected graph on~$n$ nodes with~$e$ edges. Then $\dim \mathrm{GV} (\mathcal{L}_G) \leqslant 2n+e-2$.
\end{proposition}

We conjecture that $\dim \mathrm{GV} (\mathcal{L}_G) = \min {\left(2n+e-2,\binom{n+1}{2}\right)}$. When $2n+e-2 \leqslant \binom{n+1}{2}$, the conjecture is equivalent to the statement that $\{A^2,\ldots,A^{n-1}\}\cup \{E_{ij} | (i,j) \in E(G)\} \cup \{E_{ii}|i=1,\ldots,n\}$ is a linearly independent set. Here~$E(G)$ denotes the set of edges of~$G$.
This conjecture is true when~$G$ is a tree, as seen in the next section. 

We end this section by characterizing Gibbs varieties for LSSMs that correspond to simple connected graphs on~$n \leqslant 4$ vertices. 
For~$n\leqslant 3$ we always have~$\dim \mathrm{GV}(\mathcal{L}_G) = \binom{n+1}{2}$ and therefore $\mathrm{GV}(\mathcal{L}_G)$ is the entire ambient space~$\mathbb{C}^{\binom{n+1}{2}}$.
For~$n=4$ there are $6$ non-isomorphic simple connected graphs, $2$ of which are trees.
If~$G$ is not a tree, we once again have~$\dim \mathrm{GV}(\mathcal{L}_G) = 6 = \binom{n+1}{2}$ and~$\mathrm{GV}(\mathcal{L}_G) = \mathbb{C}^{\binom{n+1}{2}}$. 
If~$G$ is a tree, then~$\mathrm{GV}(\mathcal{L}_G)$ is a hypersurface.
We discuss the defining equations of these $2$ hypersurfaces in the next section. 

\section{Sparsity Patterns Given by Trees}

Trees are an important class of graphs that give rise to LSSMs with the smallest possible dimension for a given number of nodes. It is remarkable that for such LSSMs the dimension of the Gibbs variety only depends on the number of nodes in the graph (or, equivalently, the size of the matrices), and the dependence is linear. 

\begin{theorem} \label{thm:dimgraph}
    Let~$\mathcal{L}_G$ be an LSSM given by a tree~$G$ on~$n$ nodes. Then ${\dim \mathrm{GV}(\mathcal{L}_G) =3n-3}$.
\end{theorem}

\begin{proof}
    By Proposition \ref{prop:Galois} the dimension of the~$\mathbb{Q}$-linear space spanned by the eigenvalues of~$\mathcal{L}_G$ is equal to~$n$. 
    The dimension of~$\mathcal{L}_G$ is equal to~$2n-1$, since~$G$ is a tree and therefore has~$n-1$ edges.
    It remains to compute the dimension of the $\mathcal{L}_G$-centralizer of a generic element in~$\mathcal{L}_G$. 
    Suppose~$A \in \mathcal{L}_G$. We are looking for solutions of the equation~$AY-YA = 0$, $Y \in \mathcal{L}_G$.
    This is a system of homogeneous linear equations in the unknowns~$y_{ij}$. 
    We have~$(AY-YA)_{ik} = \sum a_{ij}y_{jk} - \sum y_{ij}a_{jk}$.
    Note that since~$Y \in \mathcal{L}_G$,~$y_{ij}\neq 0$ if and only if~$(i,j)$ is an edge of~$G$ or~$i=j$.
    The same is generically true for~$a_{ij}$.
    Thus, $(AY-YA)_{ik}$ is not identically zero if and only if there exists~$j$ such that~$(i,j)$ and~$(j,k)$ are edges of~$G$ or if~$(i,k)$ is itself and edge of~$G$.
    In terms of the graph~$G$, this means that $(AY-YA)_{ik}$ is not identically zero if and only if there is a path of edge length at most~$2$ from~$i$ to~$k$.
    Since~$G$ is a tree, there is at most one such path. 
    Therefore, if~$i$ and~$k$ are connected by a path of edge length $2$ via the node~$j$, the corresponding entry of~$AY-YA$ is equal to $a_{ij} y_{jk} - a_{jk}y_{ij}$. 
    It is equal to zero if~$y_{jk}$ is proportional to~$y_{ij}$ with the coefficient~$a_{ij}/a_{jk}$ (note that~$a_{jk}$ is generically non-zero). 
    Since~$G$ is connected, we conclude that all the~$y_{ij}$ with~$i\neq j$ are proportional.
    If~$i$ and~$k$ are connected by an edge, the corresponding entry of~$AY-YA$ is equal to~$y_{ii}a_{ik} + y_{kk}a_{ik} - (a_{ii}-a_{kk})y_{ik}$.
    If it is equal to zero, then~$y_{kk}$ is a linear combination of~$y_{ii}$ and~$y_{ik}$.
    We conclude that, since~$G$ is connected and all the~$y_{ik}$ are proportional, all the~$y_{ii}$ can be expressed as linear combinations of~$y_{11}$ and just one~$y_{jk}$ with~$j\neq k$. 
    Therefore, the centralizer, which is the solution space of the considered linear system, is at most $2$-dimensional. Since it contains~$\mathrm{id}_n$ and~$A$, it is exactly two-dimensional.
    The statement of the Theorem now follows from Theorem \ref{thm:dimeq} for~$m = n$,~$d=2n-1$ and~$k=2$. 
\end{proof}

\begin{example}
    For~$n=4$ there are exactly two non-isomorphic trees, shown below. By Theorem \ref{thm:dimgraph}, the dimension of their Gibbs varieties is equal to~$9$. Therefore, these Gibbs varieties are hypersurfaces in~$\mathbb{C}^{\binom{n+1}{2}}=\mathbb{C}^{10}$. 

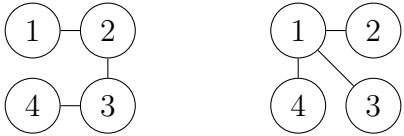
\begin{center}

\begin{tikzpicture}[main/.style = {draw, circle}] 
\node[main] (1) {$1$}; 
\node[main] (2) [ right of=1] {$2$};
\node[main] (3) [below of=2] {$3$}; 
\node[main] (4) [left of=3] {$4$};

\draw (1) -- (2);
\draw (2) -- (3);
\draw (3) -- (4);
\end{tikzpicture} 
\qquad \qquad
\begin{tikzpicture}[main/.style = {draw, circle}] 
\node[main] (1) {$1$}; 
\node[main] (2) [ right of=1] {$2$};
\node[main] (3) [below of=2] {$3$}; 
\node[main] (4) [left of=3] {$4$};

\draw (1) -- (2);
\draw (1) -- (3);
\draw (1) -- (4);
\end{tikzpicture} 
\end{center}
The corresponding LSSMs are 
$$
\begin{pmatrix}
    y_{11} &  y_{12} & 0 & 0\\
    y_{12} & y_{22} & y_{23} & 0 \\
    0 & y_{23} & y_{33} & y_{34} \\
    0 & 0 & y_{34} & y_{44}
\end{pmatrix}
\text{ and }
\begin{pmatrix}
    y_{11} & y_{12} & y_{13} & y_{14}\\
    y_{12} & y_{22} & 0 & 0\\
    y_{13} & 0 & y_{33} & 0\\
    y_{14} & 0 & 0 & y_{44}
\end{pmatrix} 
,
$$
respectively. 

For the~$4$-chain, the graph on the left, the Gibbs variety is defined by a single homogeneous equation of degree~$6$ that has $96$ terms. For the graph on the right the defining equation is also homogeneous of degree~$6$. It has~$60$ terms. These two equations were found using Algorithm \ref{alg:num}. 
\end{example}


\section{Logarithmic sparsity from coloured graphs}

Sparse LSSMs defined by coloured graphs appear in the study of coloured Gaussian graphical models in algebraic statistics \cite{HoLa}, \cite{StUhl}. In this section we study the properties of Gibbs varieties of such LSSMs. 

Consider the graph~$G$ and suppose its vertices are labeled by~$p$ colours and edges are labeled by~$q$ colours. The corresponding LSSM~$\mathcal{L}$ is cut out by the following three sets of equations.

\begin{enumerate}
    \item $x_{ij} = 0$ if $(i,j)$ is not an edge of~$G$
    \item $x_{ii} = x_{jj}$ if the vertices~$i$ and~$j$ have the same colour. 
    \item $x_{ij} = x_{kl}$ if $(i,j)$ and~$(k,l)$ are edges of~$G$ that have the same colour.
\end{enumerate}

It is immediately clear that $\dim \mathcal{L} = p+q$. 

We will denote coloured graphs by~$\mathcal{G}$ and the corresponding LSSMs by~$\mathcal{L}_\mathcal{G}$. The corresponding uncoloured graph will be denoted by~$G$, as usual. Note that since~$\mathcal{L}_{\mathcal{G}} \subseteq \mathcal{L}_{G}$, the inclusion of the Gibbs varieties also holds:~$\mathrm{GV}(\mathcal{L}_\mathcal{G}) \subseteq \mathrm{GV}(\mathcal{L}_G)$. Since the identity matrix is in~$\mathcal{L}_\mathcal{G}$ for any~$\mathcal{G}$, the dimension bound from Proposition \ref{prop:upbound} holds for coloured graphs as well. 

\begin{definition}
    We say that~$X\in \mathbb{S}_+^n$ satisfies the \emph{coloured sparsity pattern} given by~$\mathcal{G}$ if~$X \in \mathcal{L}_\mathcal{G}$. 
\end{definition}

\begin{proposition} \label{prop:dimcol}
    Let~$\mathcal{G}$ be a coloured graph on~$n$ nodes in which vertices are labeled by~$p$ colours and edges are labeled by~$q$ colours. Then~$\dim \mathrm{GV} (\mathcal{L}_\mathcal{G}) \leqslant n+p+q-2$.
\end{proposition}

Note that if~$\mathcal{G}$ is a coloured graph, the eigenvalues of~$\mathcal{L}_\mathcal{G}$ are not necessarily~$\mathbb{Q}$-linearly independent. Therefore, the upper bound from Proposition~\ref{prop:dimcol} is not always attained. 

\begin{example}
    Consider the graph 
\begin{tikzpicture}[main/.style = {draw,circle}] 
\node[main, fill=blue] (1) {}; 
\node[main, fill=blue] (2) [right of=1] {};
\node[main, fill=blue] (3) [right of=2] {};

\draw[color=red] (1) -- (2);
\draw (2) -- (3);
\end{tikzpicture}.
    The corresponding LSSM is
    $$
    \begin{pmatrix}
        y_1 & y_2 & 0\\
        y_2 & y_1 & y_3\\
        0 & y_3 & y_1
    \end{pmatrix}.
    $$
    The eigenvalues of this LSSM are~$\mathbb{Q}$-linearly dependent: they satisfy the equation~$2\lambda_1 = \lambda_2 + \lambda_3$. We have~$\dim \mathrm{GV}(\mathcal{L}) = 3 < n+p+q-2 = 3 + 1 + 2 - 2 =4$. Note that in this case~$\dim \mathrm{GV}(\mathcal{L}_\mathcal{G}) = \dim\mathrm{GM}(\mathcal{L}_\mathcal{G})$ and the Gibbs manifold of~$\mathcal{L}_\mathcal{G}$, i.e. the set of matrices that satisfy the coloured logarithmic sparsity condition given by~$\mathcal{G}$, is the positive part of its Gibbs variety, i.e.~$\mathrm{GM}(\mathcal{L}_\mathcal{G}) = \mathrm{GV}(\mathcal{L}_\mathcal{G}) \cap \mathbb{S}^n_+$. This means that the set of matrices with the coloured logarithmic sparsity pattern given by this graph can be described algebraically. 
\end{example}

In order to illustrate how different colourings of the same graph affect the Gibbs variety, we conclude this section with analysing coloured graphs for which the underlying graph is the $3$-chain. This is done using \cite[Algorithm 1]{GM}.

\begin{enumerate}
    \item  \begin{tikzpicture}[main/.style = {draw,circle}] 
\node[main, fill=blue] (1) {}; 
\node[main, fill=green] (2) [right of=1] {};
\node[main, fill=yellow] (3) [right of=2] {};

\draw[color=red] (1) -- (2);
\draw (2) -- (3);
\end{tikzpicture}

The corresponding LSSM is 
$$\mathcal{L}_\mathcal{G} = \begin{pmatrix}
    y_1 & y_4 & 0\\
    y_4 & y_2 & y_5\\
    0 & y_5 & y_3\\
\end{pmatrix}.$$
$\dim \mathrm{GV}(\mathcal{L}_\mathcal{G}) = 6$ and there are no polynomial equations that hold on the Gibbs variety.

    \item \begin{tikzpicture}[main/.style = {draw,circle}] 
\node[main, fill=blue] (1) {}; 
\node[main, fill=green] (2) [right of=1] {};
\node[main, fill=yellow] (3) [right of=2] {};

\draw (1) -- (2);
\draw (2) -- (3);
\end{tikzpicture}

The corresponding LSSM is 
$$\mathcal{L}_\mathcal{G} = \begin{pmatrix}
    y_1 & y_4 & 0\\
    y_4 & y_2 & y_4\\
    0 & y_4 & y_3\\
\end{pmatrix}.$$
$\dim \mathrm{GV}(\mathcal{L}_\mathcal{G}) = 5$ and the Gibbs variety is a cubic hypersurface whose prime ideal is generated by the polynomial
\begin{multline*}
     x_{11}x_{13}x_{23}-x_{12}^2 x_{23} + x_{12}x_{22}x_{13}-x_{12}x_{13}^2-\\-x_{12}x_{13}x_{33}+x_{12}x_{23}^2-x_{22}x_{13}x_{23}+x_{13}^2 x_{23}.
\end{multline*}

    \item \begin{tikzpicture}[main/.style = {draw,circle}] 
\node[main, fill=blue] (1) {}; 
\node[main, fill=blue] (2) [right of=1] {};
\node[main, fill=green] (3) [right of=2] {};

\draw[color=red] (1) -- (2);
\draw (2) -- (3);
\end{tikzpicture}

The corresponding LSSM is 
$$\mathcal{L}_\mathcal{G} = \begin{pmatrix}
    y_1 & y_3 & 0\\
    y_3 & y_1 & y_4\\
    0 & y_4 & y_2\\
\end{pmatrix}.$$
$\dim \mathrm{GV}(\mathcal{L}_\mathcal{G}) = 5$ and the Gibbs variety is a cubic hypersurface. Its prime ideal is generated by the polynomial 
\begin{multline*}
     -x_{11}x_{12}x_{23}+x_{11}x_{22}x_{13} - x_{11}x_{13}x_{33}+x_{12}x_{22}x_{23}-\\-x_{22}^2 x_{13}+x_{22}x_{13}x_{33}+x_{13}^3-x_{13}x_{23}^2.
\end{multline*}

    \item \begin{tikzpicture}[main/.style = {draw,circle}] 
\node[main, fill=blue] (1) {}; 
\node[main, fill=blue] (2) [right of=1] {};
\node[main, fill=green] (3) [right of=2] {};

\draw (1) -- (2);
\draw (2) -- (3);
\end{tikzpicture}

The corresponding LSSM is 
$$\mathcal{L}_\mathcal{G} = \begin{pmatrix}
    y_1 & y_3 & 0\\
    y_3 & y_1 & y_3\\
    0 & y_3 & y_2\\
\end{pmatrix}.$$
$\dim \mathrm{GV}(\mathcal{L}_\mathcal{G}) = 4$. The Gibbs variety is a complete intersection, its prime ideal is generated by the polynomials
    $$x_{11} - x_{22} + x_{33},$$
    $$-x_{12}x_{23}+x_{22}x_{13}-x_{13}^2-x_{13}x_{33}+x_{23}^2.$$

    \item \begin{tikzpicture}[main/.style = {draw,circle}] 
\node[main, fill=blue] (1) {}; 
\node[main, fill=green] (2) [right of=1] {};
\node[main, fill=blue] (3) [right of=2] {};

\draw[color=red] (1) -- (2);
\draw (2) -- (3);
\end{tikzpicture}

The corresponding LSSM is 
$$\mathcal{L}_\mathcal{G} = \begin{pmatrix}
    y_1 & y_3 & 0\\
    y_3 & y_2 & y_4\\
    0 & y_4 & y_1\\
\end{pmatrix}.$$
$\dim \mathrm{GV}(\mathcal{L}_\mathcal{G}) = 5$ and the Gibbs variety is a cubic hypersurface. Its prime ideal is generated by the polynomial 
$$-x_{11}x_{12}x_{23}+x_{12}^2 x_{13}+x_{12} x_{23} x_{33} - x_{13}x_{23}^2.$$

    \item \begin{tikzpicture}[main/.style = {draw,circle}] 
\node[main, fill=blue] (1) {}; 
\node[main, fill=green] (2) [right of=1] {};
\node[main, fill=blue] (3) [right of=2] {};

\draw (1) -- (2);
\draw (2) -- (3);
\end{tikzpicture}

The corresponding LSSM is 
$$\mathcal{L}_\mathcal{G} = \begin{pmatrix}
    y_1 & y_3 & 0\\
    y_3 & y_2 & y_3\\
    0 & y_3 & y_1\\
\end{pmatrix}.$$
$\dim \mathrm{GV}(\mathcal{L}_\mathcal{G}) = 4$ and the Gibbs variety is an affine subspace with the prime ideal generated by~$x_{12}-x_{23}$ and~$x_{11}-x_{33}$.

    \item \begin{tikzpicture}[main/.style = {draw,circle}] 
\node[main, fill=blue] (1) {}; 
\node[main, fill=blue] (2) [right of=1] {};
\node[main, fill=blue] (3) [right of=2] {};

\draw[color=red] (1) -- (2);
\draw (2) -- (3);
\end{tikzpicture}

The corresponding LSSM is 
$$\mathcal{L}_\mathcal{G} = \begin{pmatrix}
    y_1 & y_2 & 0\\
    y_2 & y_1 & y_3\\
    0 & y_3 & y_1\\
\end{pmatrix}.$$
$\dim \mathrm{GV}(\mathcal{L}_\mathcal{G}) = 3$. The prime ideal of the Gibbs variety is generated by~$7$ polynomials:
$$x_{12}x_{13}-x_{22}x_{23}+x_{23}x_{33},$$
$$x_{11}x_{13}-x_{12}x_{23}+x_{13}x_{33},$$
$$x_{11}x_{22}-x_{11}x_{33}-x_{22}^2+x_{22}x_{33}+x_{13}^2,$$
$$x_{12}^2 - x_{22}^2 + x_{13}^2 + x_{33}^2,$$
$$x_{11}x_{12}-x_{12}x_{22}+x_{13}x_{23},$$
$$x_{11}^2-x_{22}^2+x_{13}^2+x_{23}^2,$$
$$-x_{12}x_{22}x_{23}+x_{12}x_{23}x_{33}+x_{22}^2 x_{13}-x_{13}^3-x_{13}x_{33}^2.$$

    \item \begin{tikzpicture}[main/.style = {draw,circle}] 
\node[main, fill=blue] (1) {}; 
\node[main, fill=blue] (2) [right of=1] {};
\node[main, fill=blue] (3) [right of=2] {};

\draw (1) -- (2);
\draw (2) -- (3);
\end{tikzpicture}

The corresponding LSSM is 
$$\mathcal{L}_\mathcal{G} = \begin{pmatrix}
    y_1 & y_2 & 0\\
    y_2 & y_1 & y_2\\
    0 & y_2 & y_1\\
\end{pmatrix}.$$
This is a commuting family and therefore, by \cite[Theorem 2.7]{GM}, $\dim \mathrm{GV}(\mathcal{L}_\mathcal{G}) = 2$. The prime ideal of the Gibbs variety is generated by~$3$ linear forms and~$1$ quadric:~$x_{22}-x_{13}-x_{33}$,~$x_{12}-x_{23}$,~$x_{11}-x_{33}$ and~$-2x_{13}x_{33}+x_{23}^2$.
\end{enumerate}

\section{From Analytic to Algebraic Equations}
Since the logarithm is an analytic function on~$\mathbb{R}_{>0}$, the set of matrices satisfying the logarithmic sparsity pattern given by a graph~$G$ can be defined via formal power series equations. One way to write these equations in a compact form is by using Sylvester's formula. 

\begin{theorem}[Sylvester \cite{Sylv}] \label{thm:sylv}
    Let $f: D \rightarrow \mathbb{R}$ be an analytic function on an open set $D \subset \mathbb{R}$ and~$M \in \mathbb{R}^{n \times n}$ a matrix that has $n$ distinct eigenvalues~$\lambda_1,\ldots,\lambda_n $ in $D$. Then
\[ f(M) \,=\, \sum\limits_{i=1}^{n}f(\lambda_i)M_i, \quad \text{with} \quad M_i \,=\, \prod_{j\neq i}\dfrac{1}{\lambda_i-\lambda_j}(M - \lambda_j  \cdot {\rm id}_n).\]
We note that the product on the right hand side takes place
in the commutative ring $\,\RR[M]$.
\end{theorem}

By setting~$f$ to be the logarithm function, we obtain a parametrization of~$\log X$ with rational functions in the entries~$x_{ij}$ of~$X$, the eigenvalues~$\lambda_i$ of~$X$ and their logarithms~$\log \lambda_i$. The logarithmic sparsity condition induced on~$X$ requires that some components of this parametrization are zero and therefore gives a system of polynomial equations in~$x_{ij}$,~$\lambda_i$ and~$\log{\lambda_i}$. By eliminating the variables~$\lambda_i$ and~$\log \lambda_i$ from this system while taking into account the polynomial relations between~$\lambda_i$ and~$x_{ij}$ given by the coefficients of the characteristic polynomial, we obtain a set of defining equations of~$\mathrm{GV}(\mathcal{L}_G)$. This procedure is described by Algorithm \ref{alg:graph_impl}. 

\begin{algorithm}
\caption{Implicitization of the Gibbs variety of~$\mathcal{L}_G$ given by a graph~$G$} \label{alg:graph_impl}
\begin{description}[itemsep=0pt]
\item[Input ] 
    A simple undirected connected graph~$G$;
\item[Output ] A set of defining equations of~$\mathrm{GV}(\mathcal{L}_G)$. 
\end{description}

\begin{enumerate}[label = \textbf{(S\arabic*)}, leftmargin=*, align=left, labelsep=2pt, itemsep=0pt]
    \item Let~$S := \{(i,j)|1\leqslant i \leqslant j \leqslant n \text{ and } (i,j)\not\in E(G)\}$. 
    
    \item Let~$\{a_{ij}\} = A := \sum\limits_{i=1}^{n}\log(\lambda_i)X_i, \quad \text{with} \quad X_i \,=\, \prod_{j\neq i}\dfrac{1}{\lambda_i-\lambda_j}(X - \lambda_j  \cdot \mathrm{id}_n)$, where~$X = (x_{ij})$ is a symmetric matrix of variables.
    
    \item Let~$E_1 := \{a_{ij}|(i,j)\in S\}$.

    \item Clear the denominators in~$E_1$ and record the least common denominator~$D$.

    \item Compute the characteristic polynomial~$P_X(x_{ij};\lambda) = \det(X -\lambda \mathrm{id}_n) = c_0(x_{ij}) + c_1(x_{ij})\lambda+\ldots+c_n(x_{ij})\lambda^n$. 

    \item Let~$E_2 := \{\text{the $n$ polynomials } (-1)^i \sigma_{n-i}(\lambda) - c_i(y)\}$, where~$\sigma_{n-i}(\lambda)$ is the~$(n-i)$th elementary symmetric polynomial in the variables~$\lambda_1,\ldots,\lambda_n$. 
    
   \item Let~$I$ be the ideal in~$\mathbb{Q}[x_{ij},\lambda, \log \lambda]$ generated by~$E_1$ and~$E_2$. 

   \item $I := I:D^\infty$. 
   
   \item \label{step:elim} Let $J = I \cap \mathbb{Q}[x_{ij}]$. 

   \item \textbf{Return} a set of generators of~$J$.
\end{enumerate}
\end{algorithm}

\begin{theorem}
    Algorithm \ref{alg:graph_impl} is correct. The ideal~$J$ computed in step \ref{step:elim} is the prime ideal of~$\mathrm{GV}(\mathcal{L}_G)$.
\end{theorem}

\begin{proof}
    Since the eigenvalues of~$\mathcal{L}_G$ are~$\mathbb{Q}$-linearly independent, the ideal generated by~$E_2$ is prime. Moreover, there is no~$\mathbb{C}$-algebraic relation between the eigenvalues of~$X$ and their logarithms that holds for any positive definite~$X$ (this is a consequence of Ax-Schanuel theorem \cite[(SP)]{Ax1971}). These two facts ensure that all the algebraic relations between~$X$,~$\lambda$ and~$\log \lambda$ are accounted for, and that the algorithm is thus correct. 
    The ideal generated by~$E_1$ and~$E_2$ is therefore also prime, after saturation, and elimination in step~\ref{step:elim} preserves primality. 
\end{proof}

Note that the primality~$J$ means that~$\mathrm{GV}(\mathcal{L}_G)$ is irreducible, as stated in \cite[Theorem 3.6]{GM}.

The advantage of this algorithm compared to \cite[Algorithm 1]{GM} is that it uses a smaller polynomial ring and fewer variables are eliminated.

\section{Relevance in applications}

In this section we show the role logarithmically sparse matrices play in two applied contexts: high-dimensional statistics and semidefinite programming. 

A typical problem in high-dimensional statistics is estimating the covariance matrix of a random vector of length~$n$ from~$l \ll n$ samples. It is known that no consistent estimator can be derived in such setup without making additional assumptions on the structure of the covariance matrix. This problem can in some cases be solved by assuming that the covariance matrix has a fixed logarithmic sparsity pattern \cite{BatteyFirst}, \cite{BatteySecond}. An advantage of this assumption is that once a logarithmic sparsity pattern is induced on the covariance matrix~$C$, it is also automatically induced on the concentration matrix~$K = C^{-1}$, since~$(\exp{L})^{-1} = \exp{(-L)}$.
In principle, one could relax the structural assumption of logarithmic sparsity and replace it by the assumption that the covariance matrix is an element of the Gibbs variety. The advantage of such relaxation is that checking whether a given set of polynomial equations is satisfied by the matrix is generally simpler that computing the matrix logarithm and then checking whether it satisfies the sparsity condition.

We now show how logarithmically sparse matrices arise in entropic regularization of semidefinite programming \cite{GM}. We start with giving basic definitions.

We fix an arbitrary linear map $\pi : \mathbb{S}^n \rightarrow \RR^d$.
This can be written in the~form $$ \pi(X) = \bigl( \langle A_1 , X \rangle,
\langle A_2 , X \rangle, \ldots, \langle A_d , X \rangle \bigr). $$
Here the $A_i \in \mathbb{S}^n$, and $\langle A_i , X \rangle := {\rm trace}(A_i X)$.
The image $\pi(\mathbb{S}^n_+)$  of the PSD cone $\mathbb{S}^n_+$
under this linear map $\pi$ is a \emph{spectrahedral shadow}.
In our setting it is a full-dimensional semialgebraic convex cone in $\RR^d$.  

{\em Semidefinite programming (SDP)} is the 
following convex optimization problem: 
\begin{equation*} \label{eq:SDP}
    {\rm Minimize} \quad \langle C , X \rangle \quad  \hbox{subject to} \quad X \in \mathbb{S}^n_+ 
    \,\text{ and }\, \pi(X) = b. 
\end{equation*}
See e.g.~\cite[Chapter 12]{MS}.
The instance of an SDP problem is
specified by the cost matrix $C \in \mathbb{S}^n$  and the right hand side vector $b \in \RR^d$. The feasible region $\mathbb{S}^n_+ \cap \pi^{-1}(b)$ is a \emph{spectrahedron}.
The SDP problem is feasible if and only if $\,b $ is 
in~$ \pi(\mathbb{S}^n_+)$.

Consider the LSSM $\mathcal{L} = {\rm span}_\RR (A_1, \ldots, A_d)$.
We usually assume that $\mathcal{L}$ contains a positive definite matrix.
This hypothesis ensures that each spectrahedron $\pi^{-1}(b)$ is compact.

Entropic regularization of SDP \cite[Section 5]{GM} is defined as follows:
 \begin{equation*}
\label{eq:regSDP} {\rm Minimize} \quad \langle C , X \rangle \,-\, \epsilon \cdot h(X) \quad  \hbox{subject to} \quad X \in \mathbb{S}^n_+ \,\text{ and } \,\pi(X) = b. 
\end{equation*}
Here $\epsilon > 0$ is a parameter, and $h$ denotes the {\em von Neumann entropy}
$$ h \,: \, \mathbb{S}^n_+ \,\rightarrow\, \RR \,, \,\,
X \,\mapsto\, {\rm trace} \bigl( X - X \cdot {\rm log}(X)  \bigr) . $$
 
The next result appears as Theorem 5.1 in \cite{GM} and illustrates the role of Gibbs manifolds in semidefinite programming.
The following \emph{affine space of symmetric matrices} (ASSM) is obtained by incorporating $\epsilon$ and the cost matrix $C$ into
the LSSM:
 $$ \mathcal{L}_\epsilon \,\, := \,\, \mathcal{L}- \frac{1}{\epsilon} C \quad
 \hbox{for any} \,\,\epsilon > 0. $$
Here we allow the case $\epsilon = \infty$, where 
the dependency on $C$ disappears and
the ASSM is simply the LSSM, i.e.~$\mathcal{L}_\infty = \mathcal{L}$. Note that Gibbs manifolds of ASSMs are defined analogously to the case of LSSMs. 
 
\begin{theorem}
   For $b \in \pi(\mathbb{S}^n_+)$, the intersection of $\pi^{-1}(b)$ with the 
   Gibbs manifold ${\rm GM}(\mathcal{L}_\epsilon)$
   consists of a single point $X_\epsilon^*$. This point is the optimal solution to the regularized SDP.
   For $\epsilon = \infty$, it is the unique
     maximizer of von Neumann entropy on
     the spectrahedron $\pi^{-1}(b)$.
\end{theorem}

Sets of matrices satisfying a fixed logarithmic sparsity pattern are Gibbs manifolds that correspond to a particular class of SDP constraints. If the sparsity pattern is given by a graph~$G$, the spectrahedron consists of PSD matrices for which some of the entries are fixed. More precisely, the entry~$x_{ij}$ is fixed if and only if~$i=j$ or~$(i,j)$ is an edge of~$G$. If in addition the graph is coloured, then one adds the constraints~$x_{ii}=x_{jj}$ if the nodes~$i$ and~$j$ have the same colour and~$x_{ij}=x_{kl}$ if the edges~$(i,j)$ and~$(k,l)$ have the same colour. 

\begin{example}
Let~$G$ be the~$4$-chain. The corresponding LSSM is
$$
\mathcal{L}_G = 
\begin{pmatrix}
y_{11} & y_{12} & 0 & 0\\
y_{12} & y_{22} & y_{23} & 0\\
0 & y_{23} & y_{33} & y_{34}\\
0 & 0 & y_{34} & y_{44}\\
\end{pmatrix}.
$$
The spectrahedron consists of matrices 
$$
\begin{pmatrix}
   b_{11} & b_{12} & x_{13} & x_{14}\\
   b_{12} & b_{22} & b_{23} & x_{24}\\
   x_{13} & b_{23} & b_{33} & b_{34}\\
   x_{14} & x_{24} & b_{34} & b_{44}
\end{pmatrix} \in \mathbb{S}^n_+
,
$$
where the entries~$b_{ij}$ are fixed and the entries~$x_{ij}$ are arbitrary such that the matrix is PSD.
\end{example}

\section*{Acknowledgements}
The author would like to thank Bernd Sturmfels, Simon Telen and Piotr Zwernik for helpful discussions and suggestions. 

\bibliographystyle{abbrv}
\bibliography{bibliography}

\bigskip
\bigskip

\noindent
\footnotesize
{\bf Author's address:}

\smallskip

\noindent Dmitrii Pavlov,
MPI-MiS Leipzig
\hfill \url{dmitrii.pavlov@mis.mpg.de}
\end{document}